\documentclass[10pt]{amsart}
\usepackage{amssymb}

\newtheorem{thm}{Theorem}[section]
\newtheorem{lemma}[thm]{Lemma}
\newtheorem{proposition}[thm]{Proposition}

\newtheorem{definition}[thm]{Definition}
\newtheorem{corollary}[thm]{Corollary}

\newcommand{\p}{\mathbb{P}}
\newcommand{\q}{\mathbb{Q}}
\newcommand{\R}{\mathbb{R}}

\newcommand{\ot}{\mathrm{ot}}

\newcommand{\dom}{\mathrm{dom}}

\newcommand{\h}{\mathrm{ht}}

\newcommand{\restrict}{\upharpoonright}

\newcommand{\pr}{\mathrm{pr}}

\begin{document}

\title{Entangledness in Suslin lines and trees}

\author{John Krueger}

\address{John Krueger \\ Department of Mathematics \\ 
	University of North Texas \\
	1155 Union Circle \#311430 \\
	Denton, TX 76203}
\email{jkrueger@unt.edu}

\date{October 2019; revised March 2020}



\thanks{This material is based upon work supported by the National Science Foundation under Grant
	No. DMS-1464859.}

\begin{abstract}
We introduce the idea of a weakly entangled linear order, and show that it is 
consistent for a Suslin line to be weakly entangled. 
We generalize the notion of entangled linear orders to $\omega_1$-trees, and prove 
that an $\omega_1$-tree is entangled iff it is free. 
We force the existence of a Suslin tree which is $n$-entangled, but all of whose 
derived trees of dimension $n+1$ are special, for any positive $n < \omega$.
\end{abstract}

\maketitle

This article is concerned with the property of entangledness in Suslin lines and trees. 
The idea of an entangled linear order was originally introduced by Abraham and Shelah \cite{AS} 
in the context of $\omega_1$-dense sets of reals. 
Recall that an uncountable linear order $L$ is \emph{$n$-entangled}, where $n$ is a positive integer, 
if for any pairwise disjoint sequence $\langle (a_{\xi,0},\ldots,a_{\xi,n-1}) : \xi < \omega_1 \rangle$ 
of increasing $n$-tuples of $L$ and any function $g : n \to 2$, there exist $\xi, \beta < \omega_1$ 
such that for all $i < n$, $a_{\xi,i} <_L a_{\beta,i}$ 
iff $g(i) = 1$.\footnote{We say $n$-tuples $(a_{0},\ldots,a_{n-1})$ and $(b_0,\ldots,b_{n-1})$ are \emph{disjoint} if 
$a_i \ne b_j$ for all $i, j < n$.} 
And $L$ is \emph{entangled} if it is $n$-entangled for all positive integers $n$.

The concept of entangledness is closely tied to topological properties of the linear order $L$. 
For example, if $L$ is $2$-entangled then $L$ has the 
countable chain condition, and if $L$ is $3$-entangled 
then $L$ is separable. 
So any $3$-entangled dense linear order is order isomorphic 
to a set of reals. 
Todorcevic \cite{stevoentangled} proved that if there exists an entangled linear order, 
then there exist c.c.c.\ forcing posets $\p$ and $\q$ such that $\p \times \q$ is not c.c.c. 
It follows that Martin's axiom together with the negation of the continuum hypothesis implies that there does not 
exist an entangled linear order.

Recall that a Suslin line is a linear order with the 
countable chain condition which is not separable. 
By the remarks above, a Suslin line cannot be $3$-entangled. 
In this article we introduce a natural weakening of the 
property of entangledness which can consistently 
be satisfied by a Suslin line. 
For any positive integer $n$, we say that a linear order $L$ is \emph{weakly $n$-entangled} if the property 
described in the first paragraph holds, except only for pairwise disjoint sequences 
$\langle (a_{\xi,0},\ldots,a_{\xi,n-1}) : \xi < \omega_1 \rangle$ 
of increasing $n$-tuples of $L$ which have the property that there exist 
$c_0 <_L \ldots <_L c_{n-1}$ such that for all $\xi < \omega_1$ and $i < n-1$, 
$a_{\xi,i} <_L c_i <_L a_{\xi,i+1}$.

We will prove that it is consistent for a Suslin line 
to be weakly $n$-entangled for all positive integers $n$. 
Any dense c.c.c.\ linear order $L$ 
is weakly $2$-entangled iff it is $2$-entangled, so it 
is consistent for a Suslin line to be $2$-entangled. 
However, this equivalence fails if $L$ is not dense.
If $L$ is dense and separable, then $L$ is $n$-entangled iff $L$ is weakly $n$-entangled. 
Thus, we have found a natural weakening of entangledness which coincides with 
entangledness for dense separable linear orders, but can consistently be satisfied 
by Suslin lines.

It is a reasonable question to ask whether the concept of entangledness 
has any significance for partial orders other than 
linear orders. 
In this article, we introduce a natural definition  
of entangledness in the class of $\omega_1$-trees. 
Recall that an $\omega_1$-tree is a tree of height $\omega_1$ 
all of whose levels are countable, and a Suslin tree is an $\omega_1$-tree 
which has no uncountable chains or antichains. 
As is well-known, there exists a Suslin line iff there exists a Suslin tree.

Let $(T,<_T)$ be an $\omega_1$-tree. 
For any distinct nodes $x$ and $y$ of $T$, define 
$\Delta(x,y)$ as the order type of the set $\{ z \in T : z <_T x, y \}$. 
For any positive integer $n$, we say that an $\omega_1$-tree $T$ is \emph{$n$-entangled} if 
for all sequences $\langle (a_{\xi,0}, \ldots, a_{\xi,n-1}) : \xi < \omega_1 \rangle$ of 
injective $n$-tuples which satisfy that the set of ordinals 
$$
\{ \Delta(a_{\xi,i},a_{\xi,j}) : i < j < n, 
\ \xi < \omega_1 \}
$$ 
is bounded in $\omega_1$, for all $g : n \to 2$ there exist $\xi < \beta < \omega_1$ such that 
for all $i < n$, $a_{\xi,i} <_T a_{\beta,i}$ iff $g(i) = 1$. 
The restriction on $\Delta$ is required, since any Suslin tree fails to have the property without 
this restriction. 
We will see that an $\omega_1$-tree $T$ is $1$-entangled 
iff it is Suslin, and more generally, $T$ is 
$n$-entangled iff 
all of its derived trees of dimension $n$ are Suslin. 
We will also prove that for any positive integer $n$, 
it is consistent for a Suslin 
tree to be $n$-entangled, but all of its derived trees 
of dimension $n+1$ are special.

\section{Background}

A \emph{tree} is a set $T$ together with a strict partial ordering, which we will always denote by $<_T$, 
such that for any node $x$ 
of $T$, the set of nodes below $x$ is well-ordered by $<_T$. 
To simplify notation, we will identify $T$ with $(T,<_T)$. 
The order type of the set of nodes below a given node $x$ is the \emph{height} of $x$, denoted by $\h_T(x)$. 
The \emph{height} of $T$ is the least ordinal $\gamma$ such that there are no nodes in $T$ of height $\gamma$. 
Let $T_\beta$ denote the set of nodes of $T$ of height $\beta$, which is called \emph{level $\beta$ of $T$}, 
and let $T \restrict \gamma := \bigcup \{ T_\beta : \beta < \gamma \}$. 
If $x <_T y$ and $\h_T(y) = \h_T(x) + 1$, we say that 
$y$ is an \emph{immediate successor} of $x$.

Nodes $x$ and $y$ of $T$ are \emph{comparable} if 
either $x = y$, $x <_T y$, or $y <_T x$, and otherwise 
they are \emph{incomparable}. 
A \emph{chain} of $T$ is a set of pairwise comparable nodes, and an \emph{antichain} of $T$ is a 
set of pairwise incomparable nodes. 
A node of $T$ is \emph{terminal} if it has nothing above it in $T$, and otherwise is \emph{non-terminal}. 
A \emph{branch} is a maximal chain. 
A subset of $T$ is \emph{cofinal} if the heights of its nodes constitutes a cofinal subset of the height of $T$. 
A tree of height $\omega_1$ whose levels are countable is an \emph{$\omega_1$-tree}. 

A tree $T$ is \emph{normal} if (a) for every node $x$ in $T$ such that $\h_T(x)+1$ is less than 
the height of $T$, $x$ has incomparable nodes above it, 
(b) every node of $T$ has nodes comparable to it at any 
height less than the height of the tree, 
and (c) any two nodes of the same limit ordinal height have a different set of nodes below them. 
An \emph{Aronszajn tree} is an $\omega_1$-tree with no cofinal branch, or equivalently, 
no uncountable chain. 
A \emph{Suslin tree} is an $\omega_1$-tree with no uncountable chain and no uncountable antichain. 
If $T$ is a normal $\omega_1$-tree, then $T$ is Suslin iff $T$ has no uncountable antichain. 
An $\omega_1$-tree $T$ is \emph{special} if there exists a function $f : T \to \omega$ such that 
$x <_T y$ implies $f(x) \ne f(y)$. 
It is not hard to see that any special $\omega_1$-tree is an Aronszajn tree 
and is not a Suslin tree. 
A \emph{subtree} of a tree is any subset of the tree equipped with the same partial ordering. 
Any subtree is easily a tree. 
For all $a \in T$, $T_a$ denotes the subtree of $T$ consisting of all nodes $b$ such that $a \le_T b$. 
If $U$ is an uncountable subtree of $T$, then 
if $T$ is either Aronszajn, special, or Suslin, then $U$ is Aronszajn, special, or Suslin respectively. 

Let $T$ be a tree. 
If $x \in T$ and $\beta \le \h_T(x)$, let $\pr_T(x,\beta)$ denote the unique node less than or equal to $x$ 
of height $\beta$. 
Define a function $\Delta$ on pairs of nodes by 
$$
\Delta(x,y) := \ot \{ z : z \le_T x, y \},
$$
where $\ot$ denotes order type with respect to $<_T$. 
If $x \le_T y$, then $\Delta(x,y) = \h_T(x)+1$. 
If $x$ and $y$ are incomparable, 
then $\Delta(x,y)$ is the least ordinal $\beta$ less than or equal to 
the heights of $x$ and $y$ such that 
$\pr_T(x,\beta) \ne \pr_T(y,\beta)$. 
Note that if $T$ is normal then the set $\{ z : z \le_T x,y \}$ has a maximum 
element, and the height of that maximum element is equal to $\Delta(x,y) - 1$. 
In particular, if $T$ is normal then $\Delta(x,y)$ 
is a successor ordinal.

Assume that $T$ is an $\omega_1$-tree. 
The natural forcing for adding a cofinal branch of $T$ is the forcing $\p_T$ whose 
underlying set is $T$ and partially ordered by $b \le_{\p_T} a$ if $a \le_T b$ for all 
$a$ and $b$ in $T$. 
Note that a subset of $T$ is an antichain of $T$ in the sense of trees 
iff it is an antichain of $\p_T$ in the sense of forcings. 
Thus, $T$ is a Suslin tree iff $\p_T$ is $\omega_1$-c.c. 
Forcing with $T$ means taking a generic extension by the forcing $\p_T$.
 
For each positive integer $n$, let $T^n$ denote the set of all $n$-tuples 
of nodes of $T$ each of the same height, partially ordered by 
$(a_0,\ldots,a_{n-1}) <_{T^n} (b_0,\ldots,b_{n-1})$ if for all $i < n$, $a_i <_T b_i$. 
It is easy to verify that $T^n$ is an $\omega_1$-tree. 
Also, if $T$ is normal, Aronszajn, or special, then $T^n$ is normal, Aronszajn, or special respectively. 
It is not necessarily true, however, that $T$ being Suslin implies that $T^n$ is Suslin. 
By a notation introduced above, if $\vec a \in T^n$, then $T_{\vec a}$ is the subtree of $T^n$ consisting 
of all $n$-tuples $\vec b$ such that $\vec a \le_{T^n} \vec b$. 
We are mostly interested in $T_{\vec a}$ when $\vec a$ is \emph{injective}, meaning that its 
elements are all distinct. 
A tree of the form $T_{\vec a}$, where $\vec a \in T^n$ is injective, is called a 
\emph{derived tree of $T$ of dimension $n$}.

\begin{lemma}
	Assume that $T$ is a normal Suslin tree. 
	Let $2 \le n < \omega$. 
	Then for all injective $\vec a = (a_0,\ldots,a_{n-1}) \in T^n$, 
	if $T_{\vec a}$ is Suslin, then for all $m < n-1$, $T_{(a_0,\ldots,a_m)}$ forces 
	that $T_{(a_{m+1},\ldots,a_{n-1})}$ is Suslin.
\end{lemma}

\begin{proof}
	Assume that $T_{\vec a}$ is Suslin and $m < n-1$. 
	Let $\p$, $\q$, and $\R$ be the natural forcings for adding a cofinal branch of the trees 
	$T_{\vec a}$, $T_{(a_0,\ldots,a_m)}$, and $T_{(a_{m+1},\ldots,a_{n-1})}$ respectively. 
	Since $T_{\vec a}$ is Suslin, $\p$ is $\omega_1$-c.c. 
	Now a standard result in forcing is that a product $\p \times \q$ of two forcings $\p$ and $\q$ 
	is $\omega_1$-c.c.\ iff $\p$ is $\omega_1$-c.c.\ and $\p$ forces that $\q$ is $\omega_1$-c.c. 
	So it suffices to show that $\q \times \R$ is $\omega_1$-c.c.
	
	We claim that $\q \times \R$ is forcing equivalent to $\p$, and hence is $\omega_1$-c.c., which 
	finishes the proof. 
	Given $(\vec c,\vec d)$ in the product $\q \times \R$, by normality we can find $\vec e$ and 
	$\vec f$ above $\vec c$ and $\vec d$ respectively such that the elements of $\vec e$ and $\vec f$ 
	all have the same height. 
	It easily follows from this observation that the map which sends 
	$(b_0,\ldots,b_{n-1})$ to the pair $((b_0,\ldots,b_{m}), (b_{m+1},\ldots,b_{n-1}))$ is a dense 
	embedding of $\p$ into $\q \times \R$.
	\end{proof}

\begin{definition}[\cite{devlin}]
	Let $T$ be an $\omega_1$-tree. 
	For each $1 \le n < \omega$, $T$ is \emph{$n$-free} if every derived tree of $T$ 
	of dimension $n$ is Suslin. 
	The tree $T$ is \emph{free} if it is $n$-free for every $1 \le n < \omega$.
\end{definition}

Note that $T$ is $1$-free iff $T$ is Suslin. 
It is easy to see that if $1 \le n < m$ and $T$ is $m$-free, then $T$ is $n$-free.

\nocite{devlin}

\section{Suslin Lines and Trees}

In this section we will develop some additional background material which is needed for our main results. 
We will discuss Suslin lines and trees, a forcing for adding a Suslin tree with finite conditions, 
and prove some useful technical lemmas which we will use later. 
Most of the results in this section are well-known.

\begin{definition}
	Let $L$ be a linear order.
	\begin{enumerate}
		\item $L$ is \emph{dense} if for all $a <_L b$, there exists $c$ such that 
		$a <_L c <_L b$;
		\item $L$ has the \emph{countable chain condition}, or is \emph{c.c.c.}, 
		if there does not exist an uncountable collection of 
		pairwise disjoint nonempty open intervals;
		\item $L$ is \emph{separable} if there exists a countable set $D \subseteq L$ such that for all 
		$a <_L b$, if there exists $c$ such that $a <_L c <_L b$, then there exists $d \in D$ 
		such that $a <_L d <_L b$.
		\end{enumerate}
	\end{definition}

Observe that if a linear order is c.c.c., then it does not have any 
uncountable strictly increasing or decreasing sequences. 
Also, any separable linear order is c.c.c.

\begin{definition}
	A \emph{Suslin line} is a linear order which is c.c.c.\ but not separable.
	\end{definition}

Assume that $T$ is a normal Suslin tree such that the set of immediate successors of any node 
in $T$ is equipped with a linear ordering isomorphic to the rationals. 
For any node $x$, denote this ordering of the successors of $x$ by $<_x$. 
We will refer to such a tree as a tree with \emph{rational successors}. 
For any normal Suslin tree $T$ with rational successors, 
we define a linear order $S$, called the \emph{lexicographical order} of $T$, 
as follows.

Consider distinct nodes $a$ and $b$ of $T$. 
If $a <_T b$, then we define $a <_S b$, and if $b <_T a$, then we define $b <_S a$. 
Now assume that $a$ and $b$ are incomparable in $T$. 
Since $T$ is normal, there exists a largest ordinal $\gamma$ strictly less than 
the heights of $a$ and $b$ such that $x := \pr_T(a,\gamma)$ is equal to $\pr_T(b,\gamma)$, 
namely, $\gamma = \Delta(a,b) - 1$. 
Let $y := \pr_T(a,\gamma+1)$ and $z := \pr_T(b,\gamma+1)$. 
Then $y$ and $z$ are distinct immediate successors of $x$. 
Define $a <_S b$ if $y <_x z$, and $b <_S a$ if $z <_x y$. 
A routine argument which splits into many different cases shows that $S$ is transitive, 
and it is clearly total and irreflexive. 
Hence, $S$ is a linear order. 
One can also easily show that $S$ is dense.

The equivalence between the existence of a Suslin line and the existence of a 
Suslin tree was originally proved in \cite{miller}.

\begin{lemma}
	Assume that $T$ is a normal Suslin tree with rational successors. 
	Let $S$ be the lexicographical order of $T$. 
	Then $S$ is a Suslin line.
\end{lemma}

\begin{proof}
	We have already established that $S$ is a linear order. 
	To see that $S$ is not separable, consider a countable set $D \subseteq S$. 
	Then we can find $x \in T$ such that for all $a \in D$, $\h_T(a) < \h_T(x)$. 
	Pick $y$ and $z$ which are immediate successors of $x$ such that $y <_x z$. 
	Then for all $a \in T \restrict \h_T(x)$, $a <_T y$ iff $a <_T x$ iff $a <_T z$. 
	It easily follows that there is no member of $D$ which is in between $y$ and $z$.
	
	To see that $S$ is c.c.c, we first observe that for all $a <_S b$, 
	there exists $x \in T$ such that for all $y \ge_T x$, $a <_S y <_S b$. 
	Namely, if $a <_T b$, then let $x$ be an immediate successor of $a$ such that 
	$x <_a \pr_T(b,\h_T(x)+1)$. 
	Suppose that $a$ and $b$ are incomparable. 
	Choose $x$ as an 
	immediate successor of $c := \pr_T(a,\Delta(a,b)-1) = 
	\pr_T(b,\Delta(a,b)-1)$ such that 
	$$
	\pr_T(a,\Delta(a,b)) <_{c} x <_c \pr_T(b,\Delta(a,b)).
	$$
	Now assume that $\{ (a_i,b_i) : i < \omega_1 \}$ is an uncountable family of 
	pairwise disjoint open intervals of $S$. 
	For each $i < \omega_1$ choose $x_i$ such that 
	for all $y \ge_T x_i$, $a_i <_S y <_S b_i$. 
	Then the disjointness of the open intervals easily implies that 
	$\{ x_i : i < \omega_1 \}$ is an antichain of $T$, which contradicts that $T$ 
	is Suslin.
\end{proof}

\begin{lemma}
	Assume that $T$ is a normal Suslin tree with rational successors. 
	Let $S$ be the lexicographical order of $T$. 
	Suppose that $a$, $b$, and $c$ are nodes of $T$ satisfying: 
	\begin{enumerate}
		\item $\h_T(c) < \h_T(a), \h_T(b)$;
		\item $a <_S c <_S b$.
		\end{enumerate}
	Then for all $d$, if $d \le_T a,b$, then $\h_T(d) < \h_T(c)$. 
	In particular, $a$ and $b$ are incomparable in $T$.
	\end{lemma}

\begin{proof}
	We have that $a <_S c$ but it is not the case that $a <_T c$. 
	So there exists $x <_T a, c$, and immediate successors $y$ and $z$ of $x$ 
	such that $y \le_T a$, $z \le_T c$, and $y <_x z$. 
	Suppose for a contradiction that $d \le_T a,b$ but $\h_T(c) \le \h_T(d)$. 
	Since $y$ and $d$ are both less than or equal to $a$ but $y$ has height less than 
	or equal to the height of $d$, it follows that $y \le_T d$. 
	Hence $y \le_T b$. 
	It follows that $y$ witnesses that $b <_S c$, contradicting that $c <_S b$.
	\end{proof}

We now turn to developing a forcing poset for adding a Suslin tree with rational successors 
using finite conditions. 
This forcing, which is a minor variation of the forcing of 
S.\ Tennenbaum \cite{tennenbaum}, will be used in some form in all of the consistency results 
of this article. 
In addition to defining and proving the main facts about this forcing, we will also 
prove several technical lemmas which will be needed in later sections.

We fix some notation. 
Let $\q$ denote the rationals. 
For each $\alpha < \omega_1$, let 
$\q_\alpha := \{ (\alpha,q) : q \in \q \}$. 
Define $(\alpha,q) <_\alpha (\alpha,r)$ if $q < r$ in $\q$. 
With this ordering, $\q_\alpha$ is isomorphic to the rationals. 
Let $\q^* := \bigcup \{ \q_\alpha : \alpha < \omega_1 \}$. 
Define a function $h$ on $\q^*$ by $h(\alpha,q) := \alpha$ for all $(\alpha,q) \in \q^*$. 
The forcing poset $\p$ we will define introduces a Suslin tree $T$ such that 
for all $\alpha < \omega_1$, $\q_\alpha$ is equal to level $\alpha$ of $T$. 
In particular, the function $h$ will coincide with the height function $\h_T$ in the generic extension.

\begin{definition}
A condition in $\p$ is any finite tree $p$ whose nodes belong to $\q^*$ 
and satisfies that $x <_p y$ implies 
$h(x) < h(y)$.\footnote{The relation on a tree $p \in \p$ will be denoted by $<_p$. 
	Note that $h(x)$ is not the height of the node $x$ in $p$, 
	but rather its first coordinate as an ordered pair. 
	For a finite ordered set, being a tree is equivalent to its relation being 
	irreflexive, transitive, and \emph{downwards linear}, meaning that any two nodes below a 
	given node are comparable in the relation.}
Let $q \le p$ if $q$ end-extends $p$, that is, the underlying set of $p$ is a subset of the 
underlying set of $q$ and for all $x$ and $y$ in $p$, $x <_p y$ iff $x <_q y$.
\end{definition}

For a condition $p \in \p$ and an ordinal $\alpha < \omega_1$, let $p \restrict \alpha$ denote the subtree 
of $p$ consisting of nodes $x$ in $p$ such that $h(x) < \alpha$, and $p \setminus \alpha$ the subtree 
consisting of nodes $x$ 
such that $h(x) \ge \alpha$.

\begin{lemma}
	Let $p \in \p$ and $a$ and $b$ nodes of $p$ such that 
	$a$ is the immediate predecessor of $b$. 
	Let $x \in \q^*$ such that 
	$x$ is not in $p$ and satisfies that $h(a) < h(x) < h(b)$. 
	Then there exists $q \le p$ with underlying set $p \cup \{ x \}$ 
	such that $a <_q x <_q b$, and for any pair of nodes $(c,d)$ from $p$ other than 
	the pair $(a,b)$, 
	$c$ is the immediate predecessor of $d$ in $q$ iff 
	$c$ is the immediate predecessor of $d$ in $p$.
	\end{lemma}

\begin{proof}
	Let $q$ have underlying set $p \cup \{ x \}$, and 
	define $c <_q d$ if either:
	\begin{enumerate}
		\item $c <_p d$, 
		\item $c \le_p a$ and $d = x$, or
		\item $c = x$ and $b \le_p d$.
		\end{enumerate}
	A routine argument by checking cases shows that $q$ is in $\p$. 
	Clearly $q \le p$ and $a <_q x <_q b$.
	
	Consider a pair of nodes $(c,d)$ from $p$ which is different than the pair $(a,b)$. 
	We claim that $c$ is the immediate predecessor of $d$ in $p$ iff 
	$c$ is the immediate predecessor of $d$ in $q$. 
	The reverse implication is immediate from the fact that $q$ end-extends $p$. 
	Conversely, suppose that $c$ is the immediate predecessor of $d$ in $p$. 
	If $d$ is equal to $b$, then $c = a$ is the unique predecessor of $d$ in $p$, which 
	contradicts that the pair $(c,d)$ is different from the pair $(a,b)$. 
	Thus, $d \ne b$. 
	Suppose for a contradiction that $c$ is not the immediate predecessor of $d$ in $q$. 
	Since $q$ end-extends $p$ and $x$ is the only node of $q$ which is not in $p$, 
	this means that $c <_q x <_q d$. 
	By the definition of $<_q$ and since $b \ne d$, we have that $c \le_p a$ and $b <_p d$. 
	Thus, $c <_p b <_p d$, which contradicts that $c$ is the immediate predecessor of $d$ in $p$.
	\end{proof}

\begin{lemma}
	Let $p \in \p$. 
	Suppose that $\{ (a_i,b_i) : i < n \}$ is a family of distinct pairs, where $0 < n < \omega$, such that 
	for each $i < n$, $a_i$ is the immediate predecessor of $b_i$ in $p$.  
	Let $x_0, \ldots, x_{n-1}$ be distinct members of 
	$\q^*$ satisfying that for all $i < n$, 
	$h(a_i) < h(x_i) < h(b_i)$. 
	Then there exists $q \le p$ with underlying set equal to the underlying set of $p$ together with 
	$x_0,\ldots,x_{n-1}$ such that for all $i < n$, $a_i <_q x_i <_q b_i$.
	\end{lemma}

\begin{proof}
	Apply the preceding lemma inductively $n$ many times.
	\end{proof}

The next lemma describes the basic method which we will use for proving the compatibility of conditions in $\p$ 
and related forcings.

\begin{lemma}
	Let $p$ and $q$ be in $\p$, $\alpha < \beta < \omega_1$, and $n < \omega$. 
	Assume:
	\begin{enumerate}
		\item $p \restrict \alpha = q \restrict \beta$;
		\item for all $x \in p \setminus \alpha$, $h(x) < \beta$;
		\item $\{ b_i : i < n \}$ is a set of pairwise incomparable nodes in $p \setminus \alpha$ and 
		$\{ c_i : i < n \}$ is a set of distinct minimal nodes of $q \setminus \beta$;
		\item for each $i < n$, the set $\{ x \in p \restrict \alpha : x <_p b_i \}$ is equal to 
		the set $\{ y \in q \restrict \beta : y <_q c_i \}$.
		\end{enumerate}
	Then there exists $r \in \p$ which end-extends $p$ and $q$, has underlying set $p \cup q$, 
	and satisfies that for all $b \in p \setminus \alpha$ and $c \in q \setminus \beta$, 
	$b <_r c$ iff for some $i < n$, $b \le_p b_i$ and $c_i \le_q c$.\footnote{If $n = 0$, then the sets 
	described in (3) are empty, so $q$ has no relations between members of $p \setminus \alpha$ 
	and $q \setminus \beta$.}
	\end{lemma}
	
\begin{proof}
	Define $r$ with underlying set $p \cup q$ so that for all distinct $b$ and $c$ in $r$, $b <_r c$ iff 
	at least one of the following is true:
	\begin{enumerate}
		\item $b <_p c$,
		\item $b <_q c$, or
		\item for some $i < n$, $b \le_p b_i$ and $c_i \le_q c$.
		\end{enumerate}
	The relation on $r$ is irreflexive, and the proof that it is transitive and downwards 
	linear is a routine verification. 
	Clearly $r$ end-extends $p$ and $q$, and by definition has the required property described in the lemma.
\end{proof}

\begin{corollary}
	The forcing poset $\p$ is $\omega_1$-Knaster.
\end{corollary}

\begin{proof}
	Let $\langle p_\alpha : \alpha < \omega_1 \rangle$ be a sequence of distinct conditions of $\p$. 
	Since each $p_\alpha$ is finite, a straightforward thinning out 
	argument yields a stationary set $S \subseteq \omega_1$ so that for all 
	$\alpha < \beta$ in $S$, $p_\alpha \restrict \alpha = p_\beta \restrict \beta$ and 
	for all $x \in p_\alpha \setminus \alpha$, $h(x) < \beta$. 
	Now for any $\alpha < \beta$ in $S$, apply Lemma 2.8 
	(where $n = 0$ for assumption (3)) 
	to see that $p_\alpha$ and $p_\beta$ are compatible.
	\end{proof}

Consider a generic filter $G$ on $\p$. 
Define $T$ in $V[G]$ to be the set of nodes which belong to some tree in $G$, and let $b <_T a$ 
if for some $p \in G$, $b <_p a$. 
The following is a list of basic facts about $T$ which follow easily from 
the definition of $\p$ and Lemma 2.6:
\begin{enumerate}
	\item $T$ is a tree;
	\item the underlying set of $T$ is equal to $\q^*$, 
	and for all $\alpha < \omega_1$, 
	level $\alpha$ of $T$ is exactly $\q_\alpha$; 
	in particular, $T$ is an $\omega_1$-tree;
	\item $T$ is normal;
	\item $T$ has rational successors.
	\end{enumerate}
Let $\dot T$ be a $\p$-name which is forced to be the tree just described. 
Observe that by (2), $\p$ forces that the functions $h$ and $\h_{\dot T}$ coincide.

\begin{proposition}
	The forcing poset $\p$ forces that $\dot T$ is a Suslin tree.
	\end{proposition}

\begin{proof}
	Since $\p$ forces that $\dot T$ is a normal $\omega_1$-tree, it suffices to show that 
	$\p$ forces that $\dot T$ does not have an uncountable antichain. 
	Suppose for a contradiction that some condition $p$ forces that 
	$\langle \dot a_\alpha : \alpha < \omega_1 \rangle$ is an uncountable antichain of $\dot T$. 
	For each $\alpha < \omega_1$, choose $p_\alpha \le p$ and $a_\alpha$ such that 
	$a_\alpha \in p_\alpha$ and $p_\alpha$ forces that $a_\alpha = \dot a_\alpha$.

	Using a standard thinning out argument, we can find a stationary set 
	$S \subseteq \omega_1$ satisfying that for all $\alpha < \beta$ in $S$, 
	$p_\alpha \restrict \alpha = p_\beta \restrict \beta$, and for all $x \in p_\alpha \setminus \alpha$, 
	$h(x) < \beta$. 
	By Lemma 2.8, for all $\alpha < \beta$ in $S$, 
	$p_\alpha$ and $p_\beta$ are compatible in $\p$. 
	By thinning out further if necessary, 
	we may assume that for all $\alpha \in S$, 
	$h(a_\alpha) \ge \alpha$ 
	(otherwise we can find $\alpha < \beta$ for which $a_\alpha = a_\beta$, 
	which contradicts that $p_\alpha$ and $p_\beta$ are compatible). 
	Now fix a stationary set $U \subseteq S$ such that for all $\alpha < \beta$ in $U$, 
	$$
	\{ y \in p_\alpha \restrict \alpha : y <_{p_\alpha} a_\alpha \} = 
	\{ y \in p_\beta \restrict \beta : y <_{p_\beta} a_\beta \}.
	$$
	
	Consider $\alpha < \beta$ in $U$. 
	Let $a_\beta^*$ be the minimum member of $p_\beta \setminus \beta$ below $a_\beta$. 
	Applying Lemma 2.8 to $p_\alpha$ and $p_\beta$ and the sets 
	$\{ a_\alpha \}$ and $\{ a_\beta^* \}$, we can find $r \le p_\alpha, p_\beta$ 
	such that $a_{\alpha} <_r a_\beta^*$. 
	Then $a_\alpha <_r a_\beta$, which contradicts that $r$ forces that $\dot a_\alpha$ 
	and $\dot a_\beta$ are incomparable in $\dot T$.
	\end{proof}

Since $\dot T$ has rational successors, we can define the lexicographical ordering on $\dot T$ 
which is a Suslin line in $V^\p$. 
We let $\dot S$ be a $\p$-name for this Suslin line.

\begin{lemma}
	Assume that $p \in \p$, $x$, $y$, $z$, $a$, and $b$ are distinct members of $p$, and 
	\begin{enumerate}
		\item $x <_p y <_p a$;
		\item $x <_p z <_p b$;
		\item $h(y) = h(z) = h(x) + 1$;
		\item $y <_{h(x)+1} z$.
		\end{enumerate}
	Then $p$ forces that $a <_{\dot S} b$.
	\end{lemma}

\begin{proof}
	Note that $p$ forces that $y$ and $z$ are both immediate successors of $x$ in $\dot T$, and $a$ 
	and $b$ are incomparable in $\dot T$. 
	By the definition of the lexicographical ordering of $\dot T$, $p$ forces that 
	$a <_{\dot S} b$ as witnessed by $y$ and $z$.
	\end{proof}

\section{Entangledness in Suslin lines}

The notion of an entangled set of reals was introduced by Abraham and Shelah \cite{AS}. 
We can define the idea more generally for linear orders, as noted by Todorcevic 
\cite{stevoentangled}.

\begin{definition}
	Let $L$ be an uncountable linear order. For each $1 \le n < \omega$, $L$ is \emph{$n$-entangled} 
	if for any sequence 
	$$
	\langle (a_{\xi,0},\ldots,a_{\xi,n-1}) : \xi < \omega_1 \rangle
	$$
	of pairwise disjoint increasing $n$-tuples from $L$, 
	for any function $g : n \to 2$, 
	there exist $\xi, \delta < \omega_1$ 
	such that for all $i < n$, 
	$a_{\xi,i} <_L a_{\delta,i}$ iff $g(i) = 1$.
	We say that $L$ is \emph{entangled} if it is $n$-entangled for all $1 \le n < \omega$. 	
\end{definition}

We make several easy observations. 
Any uncountable 
linear order is $1$-entangled, so the notion only becomes substantial 
when $n \ge 2$. 
We can replace ``increasing $n$-tuples'' with ``injective $n$-tuples'' and get an equivalent definition. 
If $L$ is $n$-entangled, then any uncountable suborder of it is also $n$-entangled. 
If there exists a countable set $Y \subseteq L$ such that $L \setminus Y$ is $n$-entangled, then 
$L$ is $n$-entangled. 
If $1 \le m < n < \omega$, then $n$-entangled implies $m$-entangled.

We introduce some useful terminology for discussing entangledness. 
A function $g : n \to 2$ will be referred to as a \emph{type}. 
If the property described in the definition holds for 
$\xi$, $\delta$ and $g$, we say that 
the sequence of $n$-tuples \emph{realizes} the type $g$, and that the pair of $n$-tuples 
$(a_{\xi,0},\ldots,a_{\xi,n-1})$ and $(a_{\delta,0},\ldots,a_{\delta,n-1})$ \emph{satisfies} the type $g$. 
Satisfying a type is not order dependent, so the pair of $n$-tuples also satisfies the type $1 - g$.  
Hence, a sequence of $n$-tuples realizes a type $g$ iff it realizes the type $1 - g$. 
We sometimes identify a type with a sequence of $0$'s and $1$'s in the obvious way. 

Todorcevic \cite{stevoentangled} pointed out without proof that any $2$-entangled linear order is 
c.c.c., and any $3$-entangled linear order is separable. 
We prove these facts for the benefit of the reader.

\begin{proposition}
	Let $L$ be an uncountable linear order. 
	If $L$ is $2$-entangled, then $L$ is c.c.c.
	\end{proposition}

\begin{proof}
	Suppose that $L$ is not c.c.c. 
	Then there exist pairwise disjoint nonempty open intervals $I_\alpha = (a_\alpha,b_\alpha)$ 
	for $\alpha < \omega_1$. 
	If $\alpha \ne \beta$, then $a_\alpha \ne a_\beta$. 
	For otherwise, assuming that $b_\alpha < b_\beta$ for concreteness, 
	any element between $a_\alpha$ and $b_\alpha$ is in $I_\alpha \cap I_\beta$. 
	Now $b_\alpha = a_\beta$ is possible, but since it can only happen for at most one $\beta$ by the previous 
	observation, we can inductively select uncountably 
	many such intervals where this does not happen. 
	This way we may assume without loss of generality that the sequence of 
	pairs $\langle (a_\alpha,b_\alpha) : \alpha < \omega_1 \rangle$ is pairwise disjoint.
	
	Suppose for a contradiction that for some $\alpha \ne \beta$ in $\omega_1$, the pair $(a_\alpha,b_\alpha)$ and 
	$(a_\beta,b_\beta)$ satisfies the type $(1,0)$. 
	Without loss of generality, $a_\alpha < a_\beta$ and $b_\beta < b_\alpha$. 
	Then $I_\beta \subseteq I_\alpha$, 
	and since $I_\beta \ne \emptyset$, 
	$I_\alpha$ and $I_\beta$ are not disjoint, which is a contradiction. 
	Hence, $L$ is not $2$-entangled.
	\end{proof}

\begin{proposition}
	Let $L$ be an uncountable linear order. 
	If $L$ is $3$-entangled, then $L$ is separable.
	\end{proposition}

\begin{proof}
	Assume that $L$ is not separable, and we will show that $L$ is not $3$-entangled. 
	Since $3$-entangled implies $2$-entangled, if $L$ is not $2$-entangled then we are done. 
	So assume that $L$ is $2$-entangled. 
	Then $L$ is c.c.c.\ by Proposition 3.2.

	Define an injective sequence of increasing pairs 
	$\langle (a_\xi,b_\xi) : \xi < \omega_1 \rangle$ from $L$ recursively as follows. 
	Let $\delta < \omega_1$, and assume that $a_\beta$ and $b_\beta$ are defined for all $\beta < \delta$. 
	Since $L$ is not separable, the set $E_\delta := \{ a_\beta, b_\beta : \beta < \delta \}$ is not dense. 
	If there are uncountably many pairs $a <_L b$ which form a nonempty open interval which does not meet 
	$E_\delta$, then we may choose one such pair $(a_\delta, b_\delta)$ which is not equal to the pair 
	$(a_\beta,b_\beta)$ for any $\beta < \delta$. 
	If there are only countably many such pairs, then note that for any such pair $a <_L b$, the open interval with 
	endpoints $a$ and $b$ must be countable. 
	But then we could add all such pairs and anything between them to $E_\delta$ to produce a countable 
	dense set, which is a contradiction.

	We claim that each element of $L$ occurs in a pair of the form $(a_\beta,b_\beta)$ at most 
	countably often. 
	Suppose for a contradiction that $a \in L$ occurs uncountably often. 
	Then it occurs as the first or second member of such a pair uncountably often. 
	Without loss of generality assume the former, since the other case is similar.  
	So there exists an uncountable set $X \subseteq \omega_1$ such that for all $\beta \in X$, $a = a_\beta$. 
	Then for all $\beta < \xi$ in $X$, since $b_\beta$ is not in between $a_\xi$ and $b_\xi$ and 
	$b_\beta \ne b_\xi$, we 
	must have that $a_\beta = a = a_\xi <_L b_\xi <_L b_\beta$. 
	So $\langle b_\beta : \beta \in X \rangle$ is an uncountable strictly decreasing sequence of elements of $L$, 
	which contradicts that $L$ is c.c.c.

	Using the claim, we can thin out our sequence of pairs $(a_\beta,b_\beta)$ for $\beta < \omega_1$  
	to be pairwise disjoint on an uncountable set. 
	Thus, without loss of generality we may assume that 
	the sequence $\langle (a_\beta,b_\beta) : \beta < \omega_1 \rangle$ 
	is a sequence of pairwise disjoint increasing pairs from $L$. 
	So for all $\beta < \xi < \omega_1$, neither $a_\beta$ nor $b_\beta$ are in the closed 
	interval $[a_\xi,b_\xi]$.
	
	We claim that for all $\gamma < \omega_1$, there are $\gamma < \xi < \delta < \omega_1$ such that 
	$$
	a_\xi <_L a_\delta <_L b_\delta <_L b_\xi
	$$
	First, since $L$ is c.c.c., 
	there exist $\gamma < \xi < \delta < \omega_1$ such that 
	the open intervals $(a_\xi,b_\xi)$ and $(a_\delta,b_\delta)$ are not disjoint. 
	By the choice of our pairs, neither $a_\xi$ nor $b_\xi$ are in the closed interval $[a_\delta,b_\delta]$. 
	Since the intervals $(a_\xi,b_\xi)$ and $(a_\delta,b_\beta)$ 
	meet each other, obviously the only way that this can happen is that 
	$a_\xi <_L a_\delta <_L b_\delta <_L b_\xi$, as required.
	
	Using this claim, we can define by recursion 
	a sequence $\langle (\xi_i,\delta_i) : i < \omega_1 \rangle$ 
	of increasing pairs of ordinals in $\omega_1$ such that:
	\begin{enumerate}
		\item for all $i < \omega_1$, 
		$a_{\xi_i} <_L a_{\delta_i} <_L b_{\delta_i} <_L b_{\xi_i}$;
		\item for all $i < j < \omega_1$, $\delta_i < \xi_j$.
		\end{enumerate}
	For each $i < \omega_1$, let $\vec x_i := (a_{\xi_i},a_{\delta_i},b_{\xi_i})$. 
	Then $\langle \vec x_i : i < \omega_1 \rangle$ is a pairwise disjoint sequence of 
	increasing $3$-tuples. 
	We claim that this sequence does not realize the type $g := (1,0,1)$, which implies 
	that $L$ is not $3$-entangled.
	
	Consider $i < j < \omega_1$, and we will show that $\vec x_i$ and $\vec x_j$ 
	do not satisfy $g$. 
	First, assume that $a_{\xi_j} <_L a_{\xi_i}$. 
	Since $a_{\xi_i}$ is not in the closed interval $[a_{\xi_j},b_{\xi_j}]$, 
	we must have that $b_{\xi_j} <_L a_{\xi_i}$. 
	It follows that $\vec x_i$ and $\vec x_j$ satisfy the type $(0,0,0)$, and hence do not satisfy 
	the type $g$.
	
	Secondly, assume that $a_{\xi_i} <_L a_{\xi_j}$. 
	If $\vec x_i$ and $\vec x_j$ satisfy $g = (1,0,1)$, then we must have that 
	$$
	a_{\xi_i} <_L a_{\xi_j}, \ a_{\delta_j} <_L a_{\delta_i}, \ b_{\xi_i} <_L b_{\xi_j}.
	$$
	Combining this with the other information we have, it follows that 
	$$
	a_{\xi_i} <_L a_{\xi_j} <_L a_{\delta_j} <_L a_{\delta_i} <_L b_{\xi_i} <_L b_{\xi_j}.
	$$
	In particular, $b_{\xi_i}$ is in the interval $[a_{\xi_j},b_{\xi_j}]$, which is a contradiction 
	since $\xi_i < \xi_j$.
	\end{proof}

\begin{corollary}
	Any dense $3$-entangled linear order is order isomorphic to a set of reals.
	\end{corollary}

\begin{proof}
	As is well-known, any dense separable linear order is isomorphic to a set of reals.
	\end{proof}
 
\begin{corollary}
	Assume that $L$ is a Suslin line. 
	Then $L$ is not $3$-entangled.
	\end{corollary}

We now introduce a natural weakening of the notion of entangledness in linear orders 
which can be satisfied by a Suslin line. 
Recall that any linear order is $1$-entangled, so we restrict our attention to $n \ge 2$.

\begin{definition}
	Let $L$ be a linear order and $2 \le n < \omega$. 
	A sequence of pairwise disjoint increasing $n$-tuples 
	$\langle (a_{\xi,0},\ldots,a_{\xi,n-1}) : \xi \in X \rangle$ from $L$, 
	where $X \subseteq \omega_1$, 
	is \emph{separated} if there exist 
	$c_0, c_1, \cdots, c_{n-2}$ in $L$ such that:
	$$
	\forall \xi \in X \ \forall i < n-1 \ \ a_{\xi,i} <_L c_i <_L a_{\xi,i+1}.
	$$
\end{definition}

\begin{definition}
	Let $L$ be an uncountable linear order. 
	For each $2 \le n < \omega$, 
	$L$ is \emph{weakly $n$-entangled} 
	if for any separated sequence 
	$$
	\langle (a_{\xi,0},\ldots,a_{\xi,n-1}) : \xi < \omega_1 \rangle
	$$
	of pairwise disjoint increasing $n$-tuples from $L$, 
	for any function $g : n \to 2$, there exist 
	$\xi, \delta < \omega_1$ such that for all $i < n$, 
	$a_{\xi,i} <_L a_{\delta,i}$ iff $g(i) = 1$. 
	We say that $L$ is \emph{weakly entangled} if it is weakly $n$-entangled for all $2 \le n < \omega$. 	
\end{definition}

We carry over all of the terminology about entangledness which was introduced after Definition 3.1
to weak entangledness in the obvious way. 
Note that an uncountable 
linear order $L$ is weakly 
$n$-entangled iff any separated sequence of pairwise disjoint increasing 
$n$-tuples from $L$ which is indexed by some uncountable subset of $\omega_1$ realizes any type.

It turns out that weakly $2$-entangled and $2$-entangled coincide for dense c.c.c.\ linear orders. 
This equivalence will not hold, however, when $n \ge 3$. 
Namely, by Corollary 3.5 a Suslin line cannot 
be $3$-entangled, but as we will see below, 
a Suslin line can be weakly entangled.

\begin{proposition}
	Let $L$ be a dense linear order which is c.c.c. 
	Then $L$ is $2$-entangled iff $L$ is weakly $2$-entangled.
\end{proposition}

\begin{proof}
	The forward implication is obvious.
	For the reverse implication, assume that $L$ is weakly $2$-entangled, and we will show that $L$ is $2$-entangled. 
	Fix a sequence $\langle (a_{\xi},b_{\xi}) : \xi < \omega_1 \rangle$ of pairwise disjoint 
	increasing pairs from $L$, and we will show that this sequence realizes any type.
	
	We split the proof into two cases. 
	First, assume that there exists an uncountable set 
	$X \subseteq \omega_1$ and some $c \in L$ such that 
	for all $\xi \in X$, $a_{\xi} <_L c <_L b_{\xi}$. 
	In this case, the sequence $\langle (a_\xi,b_\xi) : \xi \in X \rangle$ is separated. 
	Since $L$ is weakly $2$-entangled, this sequence realizes any type, 
	and therefore so does our original sequence.

	Secondly, assume that there does not exist such a set $X$. 
	Then for any $c \in L$, there exists some $\gamma < \omega_1$ such that 
	for all $\gamma < \delta < \omega_1$, 
	$c$ is not in the open interval with endpoints $a_\delta$ and $b_\delta$. 
	Using this fact, by a straightforward recursion we can find 
	an uncountable set $Y \subseteq \omega_1$ such that for all $\xi < \delta$ in $Y$, 
	neither $a_\xi$ nor $b_\xi$ are strictly in between $a_\delta$ and $b_\delta$.
	
	Observe that for all $\xi < \delta$ in $Y$, either:
	\begin{enumerate}
		\item $a_\xi < b_\xi < a_\delta$, 
		\item $a_\xi < a_\delta < b_\delta < b_\xi$, or 
		\item $b_\delta < a_\xi < b_\xi$,
		\end{enumerate}
	with the other possibilities being excluded by the choice of $Y$. 
	In the first and third case, the types $(1,1)$ and $(0,0)$ are both satisfied 
	by $(a_\xi,b_\xi)$ and $(a_\delta,b_\delta)$. 
	In the second case, the types $(1,0)$ and $(0,1)$ are both satisfied by 
	$(a_\xi,b_\xi)$ and $(a_\delta,b_\delta)$. 
	Hence, we will be done if we can show that case 2 occurs for some $\xi < \delta$ in $Y$, 
	and either case 1 or case 3 occurs for some $\xi < \delta$ in $Y$.

	Suppose for a contradiction that case 2 does not occur. 
	Then for all $\xi < \delta$, cases 1 and 3 imply that 
	the intervals $(a_\xi,b_\xi)$ and $(a_\delta,b_\delta)$ are pairwise disjoint and nonempty 
	since $L$ is dense. 
	But this contradicts that $L$ is c.c.c. 
	Thus, indeed case 2 occurs. 
	Now assume that neither case 1 nor 3 occurs. 
	Then for all $\xi < \delta$ in $Y$, $a_\xi < a_\delta$. 
	So $L$ contains a strictly increasing sequence of order type $\omega_1$, which implies 
	that $L$ is not c.c.c., which again is a contradiction.
	\end{proof}

Let us see that we cannot remove the assumption of being dense in the above proposition.
Namely, the linear order $M$ in the next result is not dense.

\begin{proposition}
	Suppose that $L$ is a dense weakly $2$-entangled c.c.c.\ linear order. 
	Then there exists a weakly $2$-entangled c.c.c.\ linear order $M$ which is not $2$-entangled. 
	Moreover, $M$ is separable iff $L$ is separable.
	\end{proposition}

\begin{proof}
	Consider $M := L \times 2$ with the lexicographical ordering. 
	So $(a,m) <_{M} (b,n)$ iff either (1) $a <_L b$, or (2) $a = b$, $m = 0$, 
	and $n = 1$.\footnote{See \cite{linear} for more information about this linear order.} 
	Note that for any $a \in L$, there is nothing in $M$ in between $(a,0)$ and $(a,1)$. 
	Fixing any injective sequence $\langle a_i : i < \omega_1 \rangle$ of elements of $L$, 
	it is easy to verify that the sequence $\langle ( (a_i,0), (a_i,1) ) : i < \omega_1 \rangle$ fails 
	to realize the type $(1,0)$ and so witnesses that $M$ is not $2$-entangled.
	
	Let us show that $M$ is weakly $2$-entangled. 
	Consider a separated sequence of pairwise disjoint increasing pairs 
	$\langle (x_i,y_i) : i < \omega_1 \rangle$ from $M$. 
	Fix $z \in M$ such that for all $i < \omega_1$, 
	$x_i <_{M} z <_{M} y_i$. 
	For some $m, n < 2$, there are uncountably many $i < \omega_1$ such that $x_i \in L \times \{ m \}$ 
	and $y_i \in L \times \{ n \}$. 
	Without loss of generality, assume that this holds for all $i < \omega_1$. 
	For each $i < \omega_1$, 
	write $x_i = (a_i,m)$ and $y_i = (b_i,n)$. 
	Note that for all $i < \omega_1$, since $z$ is in between $x_i$ and $y_i$, we cannot have that $a_i = b_i$. 
	It easily follows that $a_i <_L b_i$ by the lexicographical ordering. 

	Write $z = (c,k)$ for some $k < 2$. 	
	Now it could be the case that for some $i < \omega_1$, 
	$c$ is equal to either $a_i$ or $b_i$. 
	However, this can happen at most twice. 
	So without loss of generality, assume that $c$ is not equal to $a_i$ or $b_i$ for all $i < \omega_1$.
	By the lexicographical ordering, we have that $a_i <_L c <_L b_i$ for all $i < \omega_1$.
	 
	Thus, $\langle (a_i,b_i) : i < \omega_1 \rangle$ is a separated sequence of pairwise disjoint 
	increasing pairs of $L$. 
	Since $L$ is weakly $2$-entangled, this sequence realizes any type. 
	By the lexicographical ordering, our original sequence 
	$\langle (x_i,y_i) : i < \omega_1 \rangle$ realizes any type as well, as can be easily verified. 
	The other claims that $M$ is c.c.c., and that $M$ is separable iff $L$ is separable, follow 
	by similar arguments which we will leave for the reader.
	\end{proof}

Our next goal is to prove that for dense separable linear orders, 
$n$-entangled and weakly $n$-entangled are equivalent.

\begin{lemma}
	Let $L$ be a dense separable linear order. 
	Then for any pairwise disjoint sequence of increasing 
	pairs $\langle (a_\xi,b_\xi) : \xi < \omega_1 \rangle$ from $L$, 
	there exists an uncountable set $X \subseteq \omega_1$ and some $c \in L$ 
	such that for all $\xi \in X$, $a_\xi <_L c <_L b_\xi$. 
	\end{lemma}

\begin{proof}
	Fix a countable dense subset $D$ of $L$. 
	For each $\xi < \omega_1$, fix $c_\xi \in D$ such that 
	$a_\xi <_L c_\xi <_L b_\xi$. 
	Since $D$ is countable, we can find $X \subseteq \omega_1$ uncountable and $c \in L$ 
	such that for all $\xi \in D$, $c_\xi = c$.
\end{proof}

\begin{proposition}
	Suppose that $L$ is a dense separable linear order. 
	Then for all $2 \le n < \omega$, $L$ is $n$-entangled iff $L$ is weakly $n$-entangled. 
	In particular, $L$ is entangled iff $L$ is weakly entangled.
\end{proposition}

\begin{proof}
	The forward direction is immediate. 
	For the reverse direction, assume that $L$ is weakly $n$-entangled, and 
	let $\langle \vec x_i : i < \omega_1 \rangle$ be a sequence of pairwise disjoint 
	increasing $n$-tuples of $L$. 
	Using Lemma 3.10 inductively $n-1$ many times, we can find an uncountable set $X \subseteq \omega_1$ 
	such that the sequence $\langle \vec x_i : i \in X \rangle$ is separated. 
	Since $L$ is weakly $n$-entangled, this sequence realizes any type, hence so does the original sequence.
	\end{proof}

\begin{proposition}
	For all $2 \le n < \omega$, an uncountable set of reals is $n$-entangled iff it is weakly $n$-entangled. 
	In particular, an uncountable set of reals is entangled iff it is weakly entangled.
	\end{proposition}

\begin{proof}
	Standard arguments about the real number line show that (1) any set of reals is 
	a separable linear order, and 
	(2) for any uncountable set of reals $X$, there is a countable set $Y \subseteq X$ such that 
	$X \setminus Y$ is a dense linear order. 
	Let $X$ be an uncountable set of reals, and assume that $X$ is weakly $n$-entangled. 
	Fix a countable set $Y \subseteq X$ such that $X \setminus Y$ is dense. 
	Then $X \setminus Y$ is a dense separable linear order which is weakly $n$-entangled, 
	so by Proposition 3.11, $X \setminus Y$ is $n$-entangled. 
	Since $Y$ is countable, it follows that $X$ is $n$-entangled as well.
\end{proof}

\begin{thm}
	There exists an $\omega_1$-Knaster forcing poset which adds a Suslin line which is weakly entangled.
	\end{thm}

\begin{proof}
	Let $\p$ be the forcing of Definition 2.5 which adds a normal Suslin tree with rational successors. 
	Recall that the conditions of $\p$ are finite trees whose elements belong to 
	$\q^* = \bigcup \{ \q_\alpha : \alpha < \omega_1 \}$, where 
	$\q_\alpha = \{ (\alpha,q) : q \in \q \}$ for each $\alpha < \omega_1$. 
	The function $h$ maps an element $(\alpha,q)$ in $\q^*$ to $\alpha$. 
	For each $p \in \p$ and $\alpha < \omega_1$, $p \restrict \alpha$ and $p \setminus \alpha$ 
	are the subtrees of $p$ consisting of nodes $x \in p$ such that $h(x) < \alpha$ or $h(x) \ge \alpha$ 
	respectively. 
	Let $\dot T$ be a $\p$-name for the generic tree and 
	let $\dot S$ be a $\p$-name for the lexicographical ordering of $T$.

	We will prove that $\p$ forces that $\dot S$ is weakly $n$-entangled for all $2 \le n < \omega$. 
	Fix $2 \le n < \omega$ and a function $g : n \to 2$. 
	Suppose that $p \in \p$, $c_0, \ldots, c_{n-2}$ are in $\q^*$, and 
	$p$ forces that 
	$$
	\langle (\dot a_{\xi,0},\ldots,\dot a_{\xi,n-1}) : \xi < \omega_1 \rangle
	$$
	is a pairwise disjoint sequence of increasing $n$-tuples of $\dot S$ such that 
	for all $i < n-1$, $\dot a_{\xi,i} <_{\dot S} c_i <_{\dot S} \dot a_{\xi,i+1}$. 
	By extending $p$ further if necessary, we may assume without loss of generality 
	that $c_0,\ldots,c_{n-2}$ are in $p$. 
	Our goal is to find $q \le p$ and $\xi < \beta < \omega_1$ such that $q$ forces that 
	$(\dot a_{\xi,0},\ldots,\dot a_{\xi,n-1})$ and 
	$(\dot a_{\beta,0},\ldots,\dot a_{\beta,n-1})$ 
	satisfy the type $g$.

	For each $\xi < \omega_1$, fix a condition $p_\xi \le p$ 
	and elements $a_{\xi,0}, \ldots, a_{\xi,n-1}$ in $p_\xi$ 
	such that $p_\xi$ forces that for all $i < n$, 
	$\dot a_{\xi,i} = a_{\xi,i}$. 
	By a straightforward thinning out argument, we can 
	find a stationary set $U \subseteq \omega_1$ 
	of limit ordinals larger than the ordinals 
	$h(c_0),\ldots,h(c_{n-2})$ 
	such that for all $\xi < \beta$ in $U$, 
	$p_\xi \restrict \xi = p_\beta \restrict \beta$, 
	and for all $x \in p_\xi$, $h(x) < \beta$.

	By Lemma 2.8, for all $\xi < \beta$ in $U$, $p_\xi$ and $p_\beta$ are compatible. 
	It follows that the $n$-tuples 
	$(a_{\xi,0},\ldots,a_{\xi,n-1})$ and $(a_{\beta,0},\ldots,a_{\beta,n-1})$ 
	must be disjoint. 
	Thinning out further if necessary, we may assume without loss of generality that:
	\begin{enumerate}
		\item for all $\xi \in U$, the elements of the $n$-tuple $(a_{\xi,0},\ldots,a_{\xi,n-1})$ 
		are in $p_\xi \setminus \xi$;
		\item for all $\xi < \beta$ in $U$ and $i < n$, 
		the sets 
		$\{ x \in p_\xi \restrict \xi : x <_{p_\xi} a_{\xi,i} \}$ and 
		$\{ x \in p_\beta \restrict \beta : x <_{p_\beta} a_{\beta,i} \}$ are equal.
		\end{enumerate}

	For each $\xi \in U$ and $i < n$, let $a_{\xi,i}^*$ denote the minimal member of 
	the tree $p_{\xi} \setminus \xi$ which is less than or equal to $a_{\xi,i}$. 
	Then $a_{\xi,i}^*$ is minimal in the tree $p_\xi \setminus \xi$. 
	Clearly, for all $\xi < \beta$ in $U$ and $i < n$, the 
	sets 
	$\{ x \in p_\xi \restrict \xi : x <_{p_\xi} a_{\xi,i} \}$ and 
	$\{ x \in p_\beta \restrict \beta : x <_{p_\beta} a_{\beta,i}^* \}$ are equal.
	
	We claim that for all $\xi \in U$ and $i < j < n$, $a_{\xi,i}^* \ne a_{\xi,j}^*$, and hence 
	$a_{\xi,i}$ and $a_{\xi,j}$ are incomparable in $p_\xi$. 
	Recall that $p_\xi$ forces that $a_{\xi,i} <_{\dot S} c_i <_{\dot S} a_{\xi,j}$. 
	If $a_{\xi,i}^* = a_{\xi,j}^*$, then $a_{\xi,i}^*$ is below 
	both $a_{\xi,i}$ and $a_{\xi,j}$ in $p_\xi$.  
	By Lemma 2.4, it follows that 
	$p_\xi$ forces that the height of $a_{\xi,i}^*$ in $\dot T$ is less than the 
	height of $c_i$ in $\dot T$. 
	This is impossible since $h(a_{\xi,i}^*) \ge \xi > h(c_i)$, and the functions 
	$h$ and $\h_{\dot T}$ are forced to be equal.
	
	Fix $\xi < \beta$ in $U$. 
	Let $y := \{ i < n : g(i) = 1 \}$. 
	Applying Lemma 2.8, we can find $q \le p_\xi, p_\beta$ whose underlying set is 
	equal to the union of the underlying sets of $p_\xi$ and $p_\beta$ 
	such that for all $b \in p_\xi \setminus \xi$ and $c \in p_\beta \setminus \beta$, 
	$b <_{q} c$ iff for some $i \in y$, 
	$b \le_{p_\xi} a_{\xi,i}$ and $a_{\beta,i}^* \le_{p_\beta} c$. 
	In particular, for all $i < n$ with $g(i) = 1$, 
	$a_{\xi,i} <_{q} a_{\beta,i}$. 
	By the definition of the lexicographical order $\dot S$, 
	$q$ forces that $a_{\xi,i} <_{\dot S} a_{\beta,i}$.

	To complete the proof, it suffices to find $r \le q$ which forces that for all 
	$i \in n \setminus y$, $a_{\beta,i} <_{\dot S} a_{\xi,i}$. 
	As shown above, for all $i < j < n$, $a_{\beta,i}^* \ne a_{\beta,j}^*$, and for each $i < n$, 
	$a_{\xi,i}^*$ and $a_{\beta,i}^*$ have the same immediate predecessor $d_i$ in 
	$p_\xi \restrict \xi = p_\beta \restrict \beta$. 
	So the set 
	$$
	\{ (d_i, a_{\xi,i}^*) : i \in n \setminus y \} \cup 
	\{ (d_i, a_{\beta,i}^* ) : i \in n \setminus y \}
	$$
	is a finite family of distinct pairs such that  
	the first member of each pair is the immediate predecessor in the tree $q$ 
	of the second member. 
	Also, since $\xi$ is a limit ordinal, 
	$h(d_i) < \xi \le h(a_{\xi,i}^*), h(a_{\beta,i}^*)$ for all $i \in n \setminus y$.
	
	Since $q$ is finite and the rationals are infinite, 
	we can find a family of distinct rationals 
	$$
	\{ q_{\delta,i} : \delta \in \{ \beta, \xi \}, i \in n \setminus y \}
	$$ 
	such that $q_{\beta,i} < q_{\xi,i}$ for all $i \in n \setminus y$, 
	and none of these rationals appear in $q$. 
	Let $x_i := (h(d_i)+1,q_{\beta,i})$ and 
	$y_i := (h(d_i)+1,q_{\xi,i})$ be in $\q_{h(d_i)+1}$ for all $i \in n \setminus y$. 
	By Lemma 2.7, we can find $r \le q$ such that for all 
	$i \in n \setminus y$, $d_i <_r x_i <_r a_{\beta,i}^*$ and 
	$d_i <_r y_i <_r a_{\xi,i}^*$. 
	Then for each $i \in n \setminus y$, $d_i <_r x_i <_r a_{\beta,i}$, 
	$d_i <_r y_i < a_{\xi,i}$, $x_i <_{\h(d_i)+1} y_i$, and $h(x_i)$ and $h(y_i)$ are equal to $h(d_i) + 1$. 
	By Lemma 2.11, for all $i \in n \setminus y$, $r$ forces that 
	$a_{\beta,i} <_{\dot S} a_{\xi,i}$.
	\end{proof}

\section{Entangledness in Trees}

It is reasonable to ask whether the idea of entangledness has any significance for 
partial orders other than linear orders. 
In this section we address this question in the class of $\omega_1$-trees. 

Consider an $\omega_1$-tree $T$. 
Roughly speaking, $T$ being $n$-entangled should mean that any pairwise disjoint 
sequence of $n$-tuples $\langle (a_{\xi,0},\ldots,a_{\xi,n-1}) : \xi < \omega_1 \rangle$ of nodes of $T$, 
satisfying some property, realizes in some sense any type $g : n \to 2$. 
We will need to specify a reasonable property of sequences to assume 
and to define what it means to satisfy a given type. 
Note that any pairwise disjoint sequence of $n$-tuples has an uncountable subsequence such that any 
node in one of the $n$-tuples has height less than the nodes appearing in $n$-tuples with larger index. 
But in that case, the only possible relations between 
a node $a$ in one $n$-tuple and a node $b$ in another with larger height is that either 
$a <_T b$ or $a \not <_T b$. 
Thus, for satisfying a type, 
it makes sense to indicate $a <_T b$ with the integer $1$ and $a \not <_T b$ with the integer $0$. 

A natural definition of $T$ being $1$-entangled would be that for any sequence  
$\langle a_\xi : \xi < \omega_1 \rangle$ 
of nodes in $T$, where $\h_T(a_\xi) < \h_T(a_\beta)$ for all $\xi < \beta < \omega_1$, 
there exist $\xi < \beta < \omega_1$ such that $a_{\xi} <_T a_{\beta}$, and 
there exist $\gamma < \delta$ such that $a_{\gamma} \not <_T a_\delta$. 
Thus, $T$ is $1$-entangled iff $T$ has no uncountable antichains and no uncountable chains, in other 
words, iff $T$ is Suslin.
Thus, all further discussion of entangledness in $\omega_1$-trees will be restricted to Suslin trees. 

Suppose that the tree $T$ is Suslin. 
Let us consider $n$-entangledness in $T$ for $n > 1$. 
We first observe that, since $T$ is Suslin, 
for all $\gamma < \omega_1$ there exists a node $x$ with height $\gamma$ which has 
two incomparable nodes above it. 
For suppose otherwise that $\gamma$ is a counterexample. 
Since the levels of $T$ are countable, we can find 
uncountably many nodes of height greater than $\gamma$ which have 
the same node below them on level $\gamma$. 
By the choice of $\gamma$, any two such nodes are comparable, and hence they form 
an uncountable chain, contradicting that $T$ is Suslin.

The following argument shows that some additional restriction must be placed on the 
sequences referred to in the definition of $2$-entangled. 
Using the observation from the previous paragraph, we can build by recursion a sequence of triples 
$\langle (x_i,y_i,z_i) : i < \omega_1 \rangle$ satisfying:
\begin{enumerate}
	\item for all $i < \omega_1$, 
	$x_i <_T y_i$, $x_i <_T z_i$, and $y_i$ and $z_i$ are incomparable;
	\item for all $i < j < \omega_1$, 
	$\h_T(y_i)$ and $\h_T(z_i)$ are less than $h_T(x_j)$.
	\end{enumerate}
Now consider the pairwise disjoint sequence $\langle (y_\xi,z_\xi) : \xi < \omega_1 \rangle$. 
We claim that this sequence does not realize the type $(1,1)$. 
Consider $\xi < \beta < \omega_1$, and suppose that $y_\xi <_T y_\beta$. 
Since $x_\beta <_T y_\beta$ and $\h_T(y_\xi) < \h_T(x_\beta)$, 
it follows that $y_\xi <_T x_\beta$. 
Since $x_\beta <_T z_\beta$, we have that $y_\xi <_T z_\beta$. 
But $y_\xi$ and $z_\xi$ are incomparable, so they cannot both be below the node $z_\beta$. 
Hence, $z_\xi \not <_T z_\beta$. 
Thus, the type $(1,1)$ is not realized.

Note that in the above example, for all $\xi < \omega_1$, $\Delta(y_\xi,z_\xi) = \h_T(x_\xi)+1$. 
In particular, the set $\{ \Delta(y_\xi,z_\xi) : \xi < \omega_1 \}$ is unbounded in $\omega_1$. 
Thus, the following definition of $n$-entangledness excludes this example.

\begin{definition}
	Let $1 \le n < \omega$. 
	An $\omega_1$-tree $T$ is \emph{$n$-entangled} if 
	for all sequences $\langle (a_{\xi,0}, \ldots, a_{\xi,n-1}) : \xi < \omega_1 \rangle$ of 
	injective $n$-tuples which satisfy that the set 
	$$
	\{ \Delta(a_{\xi,i},a_{\xi,j}) : i < j < n, 
	\ \xi < \omega_1 \}
	$$ 
	is bounded in $\omega_1$, for all $g : n \to 2$ there exist $\xi < \beta < \omega_1$ such that 
	for all $i < n$, $a_{\xi,i} <_T a_{\beta,i}$ iff $g(i) = 1$. 
	The tree $T$ is \emph{entangled} if it is $n$-entangled for all $1 \le n < \omega$. 
	\end{definition}

We carry over all of the terminology about entangledness in linear orders which was introduced after Definition 3.1
to entangledness in $\omega_1$-trees in the obvious way. 
In addition, we introduce the following.

\begin{definition}
	Let $1 \le n < \omega$, $X \subseteq \omega_1$, and 
	$\langle (a_{\xi,0}, \ldots, a_{\xi,n-1}) : \xi \in X \rangle$ 
	a sequence of injective $n$-tuples of nodes in an $\omega_1$-tree $T$. 
	The sequence has \emph{increasing height} 
	if for all $\xi < \beta$ in $X$, for all $i, j < n$, 
	$\h_T(a_{\xi,i}) < \h_T(a_{\beta,j})$. 
	For $\vec c = (c_0,\ldots,c_{n-1}) \in T^n$, the sequence is 
	\emph{above $\vec c$} if for all $\xi \in X$ and $i < n$, $c_i <_T a_{\xi,i}$.
\end{definition}

The following characterization of $n$-entangledness in $\omega_1$-trees will prove useful.

\begin{lemma}
	Let $1 \le n < \omega$. 
	An $\omega_1$-tree $T$ is $n$-entangled iff for any  
	injective $\vec c \in T^n$, 
	for any sequence 
	$\langle (a_{\xi,0},\ldots,a_{\xi,n-1}) : 
	\xi < \omega_1 \rangle$ of $n$-tuples of nodes of $T$ 
	above $\vec c$ with increasing height, 
	for all $g : n \to 2$ there exist $\xi < \beta < \omega_1$ 
	such that for all $i < n$, $a_{\xi,i} <_T a_{\beta,i}$ iff $g(i) = 1$.
	\end{lemma}

\begin{proof}
	For the forward implication, assume that $T$ is $n$-entangled, and consider a sequence 
	$\langle (a_{\xi,0},\ldots,a_{\xi,n-1}) : \xi < \omega_1 \rangle$ of $n$-tuples of nodes of $T$ 
	with increasing height and above some injective $\vec c \in T^n$. 
	Recall that the elements of $\vec c$ all have the same height. 
	So for all $\xi < \omega_1$ and $i < j < n$, $c_i$ is below $a_{\xi,i}$ but 
	not below $a_{\xi,j}$, and therefore $\Delta(a_{\xi,i},a_{\xi,j})$ 
	is less than or equal to the height of $c_i$. 
	Hence, the set 
	$\{ \Delta(a_{\xi,i},a_{\xi,j}) : i < j < n, 
	\ \xi < \omega_1 \}$ is bounded by the height of 
	$\vec c$ in $T^n$. 
	Since $T$ is $n$-entangled, this sequence realizes any type.
	
	For the reverse implication, assume that the second statement of the lemma holds, 
	and suppose that 
	$\langle (a_{\xi,0}, \ldots, a_{\xi,n-1}) : \xi < \omega_1 \rangle$ is a sequence of 
	injective $n$-tuples which satisfies that for some $\delta < \omega_1$, 
	every element of the set 
	$$
	\{ \Delta(a_{\xi,i},a_{\xi,j}) : i < j < n, 
	\ \xi < \omega_1 \}
	$$ 
	is less than $\delta$. 
	Then for each $\xi < \omega_1$, the $n$-tuple 
	$$
	(\pr_T(a_{\xi,0},\delta),\ldots,\pr_T(a_{\xi,n-1},\delta))
	$$
	is an injective member of $T^n$ of height $\delta$. 
	Now fix $\vec c \in T^n$ and an uncountable set $X \subseteq \omega_1$ 
	such that for all $\xi \in X$, 
	$$
	\vec c = (\pr_T(a_{\xi,0},\delta),\ldots,\pr_T(a_{\xi,n-1},\delta)).
	$$
	Thinning out further if necessary, we may assume that the sequence indexed by $X$ has 
	increasing height. 
	By reindexing if necessary, we may assume that $X = \omega_1$. 
	By assumption, this sequence realizes any type, and hence 
	our original sequence realizes any type.
\end{proof}

Suppose that $T$ is a Suslin tree. 
Let $2 \le n < \omega$. 
Fix an injective $\vec c \in T^n$, and consider the derived tree $T_{\vec c}$. 
If $T_{\vec c}$ is not Suslin, then there exists a sequence 
$\langle (a_{\xi,0},\ldots,a_{\xi,n-1}) : \xi < \omega_1 \rangle$ of distinct nodes 
in $T_{\vec c}$ such that for all $\xi < \beta < \omega_1$, 
$(a_{\xi,0},\ldots,a_{\xi,n-1})$ and $(a_{\beta,0},\ldots,a_{\beta,n-1})$ 
are not comparable in $T_{\vec c}$. 
By thinning out if necessary, we may assume that this sequence has increasing height. 
So we get a sequence of $n$-tuples above $\vec c$ with increasing height 
which does not realize the type $g : n \to 2$, where $g(i) = 1$ for all $i < n$. 
It follows that any $\omega_1$-tree which is $n$-entangled is $n$-free, meaning that all of its 
derived trees of dimension $n$ are Suslin. 
In particular, if $T$ has a derived tree with dimension $n$ which is special, then 
$T$ is not $n$-entangled.

\begin{thm}
	Let $1 \le n < \omega$. 
	Then there exists an $\omega_1$-Knaster forcing 
	which adds a normal Suslin tree which is $n$-entangled and whose derived 
	trees with dimension $n+1$ are all special. 
	In particular, $n$-entangled does not imply $(n+1)$-entangled.
	\end{thm}

\begin{proof}
	The forcing we will use is a modification of the forcing $\p$ from Definition 2.5, and we 
	assume that the reader is familiar with the notation used in Section 2 about this forcing. 
	Let $p \in \p$. 
	For any $1 \le m < \omega$, let $p^m$ denote the set of all $m$-tuples $\vec a = (a_0,\dots,a_{m-1})$ 
	of nodes of $p$ such that for all $i < j < m$, $h(a_i) = h(a_j)$. 
	If $h(a_i) = \xi$ for all $i < m$, we will write $h(\vec a) = \xi$. 
	For $\vec a = (a_0,\dots,a_{m-1})$ and $\vec b = (b_0,\ldots,b_{m-1})$ in $p^m$, 
	we will write $\vec a <_p \vec b$ to indicate that 
	for all $i < m$, $a_i <_p b_i$.
	
	Define a forcing $\p^*$ as follows. 
	A condition in $\p^*$ is any pair $(p,f)$ satisfying:
	\begin{enumerate}
		\item $p \in \p$;
		\item $f$ is a function with finite domain and range a subset of $\omega$ such that 
		every member of its domain is a pair of the form 
		$(\vec c,\vec a)$, where $\vec c$ and $\vec a$ are injective tuples in $p^{n+1}$ and 
		$\vec c <_p \vec a$;
		\item if $(\vec c,\vec a)$ and $(\vec c,\vec b)$ are both in $\dom(f)$ and 
		$\vec a <_p \vec b$, then $f(\vec c,\vec a) \ne f(\vec c,\vec b)$.
		\end{enumerate}
	Let $(q,g) \le (p,f)$ if $q \le_{\p} p$ and $f \subseteq g$.

	For any condition $(p,f) \in \p^*$ and ordinal $\xi < \omega_1$, 
	let $f \restrict \xi$ be the restriction of $f$ to pairs of the form 
	$(\vec c,\vec a)$ where $h(\vec a) < \xi$.

	Let $\dot T$ be the $\p^*$-name for the 
	tree obtained as the union of all trees $p$ for which there exists 
	some $f$ such that $(p,f)$ is in the generic filter. 
	Let $\dot F$ be a $\p^*$-name for the union of all functions $f$ such that for some $p$, 
	$(p,f)$ is in the generic filter. 
	The following statements are true in $V^\p$:
	\begin{enumerate}
			\item $\dot T$ is a tree;
			\item the underlying set of $\dot T$ is equal to $\q^*$, 
			and for all $\alpha < \omega_1$, level $\alpha$ of $\dot T$ is equal to $\q_\alpha$;
			in particular, $\dot T$ is an $\omega_1$-tree;
			\item $\dot T$ is normal;
			\item the domain of $\dot F$ is equal to 
			the set of all pairs $(\vec c,\vec a)$ 
			such that $\vec c$ is an injective member of 
			$\dot T^{n+1}$ and $\vec a \in \dot T_{\vec c}$;
			\item every derived tree of $\dot T$ 
			of dimension $n+1$ is special. 
		\end{enumerate}
	The proofs of (1)--(3) are basically the same as in the analogous claims for the forcing $\p$ using Lemma 2.6. 
	(4) follows in an easy way from the fact that conditions are finite. 
	(5) is true because as a result of (4) and the definition of $\p^*$, 
	for any injective $\vec c \in \dot T^{n+1}$, 
	the function $\dot F \restrict (\{ \vec c \} \times \dot T_{\vec c})$ 
	witnesses that $\dot T_{\vec c}$ is special. 
	In addition, a proof similar to the argument that $\p$ is $\omega_1$-Knaster 
	shows that $\p^*$ is $\omega_1$-Knaster (also, this argument is implicit in what follows).

	It remains to show that $\p^*$ forces that $\dot T$ is $n$-entangled. 
	Fix a function $g : n \to 2$. 
	Suppose that $(p,f)$ is a condition in $\p^*$, 
	$\vec c = (c_0,\ldots,c_{n-1})$ is injective, and $(p,f)$ forces that 
	$\langle (\dot a_{\xi,0},\ldots,\dot a_{\xi,n-1}) : \xi < \omega_1 \rangle$ 
	is a sequence of $n$-tuples of $\dot T$ above $\vec c$ with increasing height. 
	By extending $(p,f)$ further if necessary, we may assume that $c_0,\ldots,c_{n-1}$ are in $p$. 
	We will find $(q,f^*) \le (p,f)$ and $\xi < \beta < \omega_1$ such that $(q,f^*)$ forces that 
	$(\dot a_{\xi,0},\ldots,\dot a_{\xi,n-1})$ and 
	$(\dot a_{\beta,0},\ldots,\dot a_{\beta,n-1})$ 
	satisfy the type $g$.
	
	For each $\xi < \omega_1$, fix $(p_\xi,f_\xi) \le (p,f)$ 
	and $a_{\xi,0}, \ldots, a_{\xi,n-1}$ in $p_\xi$ 
	such that $(p_\xi,f_\xi)$ forces that for all $i < n$, 
	$\dot a_{\xi,i} = a_{\xi,i}$. 
	By a straightforward thinning out argument, we can 
	find a stationary set $S \subseteq \omega_1$ 
	of limit ordinals such that for all $\xi < \beta$ in $S$, 
	$p_\xi \restrict \xi = p_\beta \restrict \beta$, 
	$f_\xi \restrict \xi = f_\beta \restrict \beta$, 
	and for all $x \in p_\xi$, $h(x) < \beta$.
	
	Consider $\xi < \beta$ in $S$. 
	By Lemma 2.8, there exists $q \le_{\p} p_\xi, p_\beta$ such that $q$ 
	contains no relations between members of $p_\xi \setminus \xi$ and $p_\beta \setminus \beta$. 
	It easily follows that the pair 
	$(q,f_\xi \cup f_\beta)$ is in $\p^*$ and is below 
	$(p_\xi,f_\xi)$ and $(p_\beta,f_\beta)$. 
	Hence, $(p_\xi,f_\xi)$ and $(p_\beta,f_\beta)$ are compatible in $\p^*$. 
	It follows that the $n$-tuples 
	$(a_{\xi,0},\ldots,a_{\xi,n-1})$ and $(a_{\beta,0},\ldots,a_{\beta,n-1})$ must be disjoint. 
	Thinning out further if necessary, we may assume without loss of generality that:
	\begin{enumerate}
		\item for all $\xi \in S$, the nodes of $(a_{\xi,0},\ldots,a_{\xi,n-1})$ 
		are in $p_\xi \setminus \xi$;
		\item for all $\xi < \beta$ in $S$ and $i < n$, the sets 
		$\{ x \in p_\xi \restrict \xi : x <_{p_\xi} a_{\alpha,i} \}$ and 
		$\{ x \in p_\beta \restrict \beta : x <_{p_\beta} a_{\beta,i} \}$ are equal.
	\end{enumerate}
	
	For each $\xi \in S$ and $i < n$, let $a_{\xi,i}^*$ denote the minimal member of 
	the tree $p_{\xi} \setminus \xi$ which is less than or equal to $a_{\xi,i}$. 
	Note that $a_{\xi,i}^*$ is minimal in the tree $p_\xi \setminus \xi$. 
	Clearly, for all $\xi < \beta$ in $S$ and $i < n$, the sets 
	$\{ x \in p_\xi \restrict \xi : x <_{p_\xi} a_{\xi,i}^* \}$ and 
	$\{ x \in p_\beta \restrict \beta : x <_{p_\beta} a_{\beta,i}^* \}$ are equal. 
	For all $\xi \in S$ and $i < j < n$, $a_{\xi,i}^*$ is above $c_i$ and 
	$a_{\xi,j}^*$ is above $c_j$, and $c_i \ne c_j$ have the same height. 
	Thus, $a_{\xi,i}^* \ne a_{\xi,j}^*$. 
	In particular, $a_{\xi,i}$ and $a_{\xi,j}$ are not comparable in $p_\xi$.

	Fix $\xi < \beta$ in $S$. 
	Define $y := \{ i < n : g(i) = 1 \}$. 
	Applying Lemma 2.8, we can find $q \le p_\xi, p_\beta$ in $\p$ whose underlying set is 
	the union of the underlying sets of $p_\xi$ and $p_\beta$ such that  
	for all $b \in p_\xi \setminus \xi$ and $c \in p_\beta \setminus \beta$, 
	$b <_{q} c$ iff for some $i \in y$, 
	$b \le_{p_\xi} a_{\xi,i}$ and $a_{\beta,i}^* \le_{p_\beta} c$. 
	In particular, for all $i < n$ with $g(i) = 1$, 
	$a_{\xi,i} <_{q} a_{\beta,i}$. 
	This description of $<_q$ also 
	implies that for all $i < n$ with $g(i) \ne 1$, $a_{\xi,i} \not <_q a_{\beta,i}$.	
	
	Define $f^* := f_\xi \cup f_\beta$. 
	We claim that $(q,f^*)$ is in $\p^*$. 
	It then easily follows that $(q,f^*)$ forces that for all $i < n$, 
	$a_{\xi,i} <_{\dot T} a_{\beta,i}$ iff $g(i) = 1$. 
	We already know that $q \in \p$. 
	Since $f_\xi \restrict \xi = f_\beta \restrict \beta$, $f^*$ is a function of the appropriate form. 
	Now assume that $(\vec c,\vec a)$ and $(\vec c,\vec b)$ are in the domain of $f^*$ and 
	$\vec a <_q \vec b$. 
	We claim that $f^*(\vec c,\vec a) \ne f^*(\vec c,\vec b)$.
	
	If $(\vec c,\vec a)$ and $(\vec c,\vec b)$ are either both in $\dom(f_\xi)$ or both in 
	$\dom(f_\beta)$, then we are done since $(p_\xi,f_\xi)$ and $(p_\beta,f_\beta)$ are conditions. 
	But the only way that this would not be true is if 
	$\vec c$ is in the common part $(p_\xi \restrict \xi)^{n+1} = (p_\beta \restrict \beta)^{n+1}$, 
	$\vec a \in (p_\xi \setminus \xi)^{n+1}$, and $\vec b \in (p_\beta \setminus \beta)^{n+1}$. 
	We will show that this is impossible.
	
	Let $\vec c = (c_0,\ldots,c_n)$, $\vec a = (a_0,\ldots,a_n)$, and $\vec b = (b_0,\ldots,b_n)$. 
	Then $c_i <_q a_i <_q b_i$ for all $i \le n$. 
	Reviewing the definition of $q$, for any $j \le n$ the only way that $b_j$ 
	could be above a member of $p_\xi \setminus \xi$ 
	is that for some $i_j \in y$, $a_{\beta,i_j}^* \le_{p_\beta} b_j$. 
	Since $y$ has size at most $n$, by the pigeon hole principle 
	there must be $j < k < n$ such that $i_j = i_k$. 
	Now $c_j$ and $a_{\beta,i_j}^*$ are both below $b_j$, so they are comparable in $q$. 
	But $h(a_{\beta,i_j}^*) \ge \beta$ and $h(c_j) < \beta$, so $c_j <_q a_{\beta,i_j}^*$. 
	Similarly, $c_k <_q a_{\beta,i_k}^*$. 
	But $i_j = i_k$, so $c_j$ and $c_k$ are below the same node in $q$. 
	This is impossible since $c_j$ and $c_k$ are distinct and have the same height.
\end{proof}

\section{Freeness and entangledness}

Recall from the previous section 
that an $\omega_1$-tree $T$ is $1$-entangled iff it is Suslin, and 
for all $2 \le n < \omega$, if $T$ is $n$-entangled 
then every derived tree of $T$ with dimension $n$ is Suslin. 
In other words, $n$-entangled implies $n$-free. 
In this section we will prove that $n$-entangled is equivalent to $n$-free.

For any tree $T$, let $\p_T$ denote the partial order with underlying set equal to 
the underlying set of $T$, with the partial ordering $b \le_{\p_T} a$ if $a \le_T b$.

\begin{lemma}
	Let $T$ be a Suslin tree. 
	If $A \subseteq T$ is uncountable, then there exists $a \in T$ such that $T_a$ is uncountable and 
	$A$ is dense below $a$ in $\p_T$.
	\end{lemma}

\begin{proof}
	The fact that $T$ is Suslin easily implies that there are at most countably many 
	terminal nodes of $T$. 
	So if the conclusion fails, then for all $a \in T$ with large enough height, 
	there exists $b >_T a$ 
	such that for all $c \in A$, $c \not >_T b$. 
	Define by recursion nodes $a_i$ and $b_i$ all $i < \omega_1$ as follows. 
	Let $\delta < \omega_1$ and assume that 
	$a_i$ and $b_i$ have been defined for all $i < \delta$. 
	Since $A$ is uncountable, choose $a_\delta \in A$ whose height in $T$ is greater 
	than the heights of $a_i$ and $b_i$ for all $i < \delta$. 
	Now choose $b_\delta >_T a_\delta$ such that for all $c \in A$, $c \not >_T b_\delta$. 

	We claim that $\{ b_i : i < \omega_1 \}$ is an antichain, which contradicts that $T$ is Suslin. 
	Consider $i < j$, and assume for a contradiction that $b_i$ and $b_j$ are comparable. 
	Since $\h_T(b_j) > \h_T(a_j) > \h_T(b_i)$, it follows that $b_i <_T b_j$. 
	As $a_j <_T b_j$ and $\h_T(a_j) > \h_T(b_i)$, we have that $b_i <_T a_j$. 
	But $a_j \in A$, so this contradicts the choice of $b_i$.
\end{proof}

\begin{lemma}
	Suppose that $T$ is a normal Aronszajn tree 
	such that $\p_T$ preserves $\omega_1$. 
	If $A \subseteq T$ is uncountable, $a \in T$, and 
	$A$ is dense below $a$ in $\p_T$, 
	then $a$ forces in $\p_T$ that $A$ contains 
	both an uncountable chain and an 
	uncountable antichain of $T$.
\end{lemma}
	
	\begin{proof}
	We first claim that $a$ forces that $\dot G \cap A$ is uncountable. 
	Otherwise there exists $b \ge_T a$ and $\gamma < \omega_1$ such that $b$ forces 
	that $\dot G \cap A \subseteq T \restrict \gamma$. 
	Since $T$ is normal, fix $c >_T b$ with $\h_T(c) \ge \gamma$. 
	As $A$ is dense below $a$, fix $d >_T c$ in $A$. 
	Then $d$ forces that $d \in \dot G \cap A$, which contradicts the choice of $\gamma$.
	
	Secondly, we claim that every member of $A$ which is above $a$ in $T$ has two incomparable 
	nodes above it in $A$. 
	If not, then there exists $b \in A$ above $a$ 
	such that the set $\{ c \in A : b <_T c \}$ is a chain. 
	Now $b$ forces that $\dot G \cap A$ is uncountable by the first claim. 
	And $b$ forces that every member of $\dot G \cap A$ with height greater than $\h_{T}(b)$ is above $b$. 
	It follows that $b$ forces that the set $\{ c \in A : b <_T c \}$ contains the uncountable set 
	$\dot G \cap A$. 
	Hence, the chain $\{ c \in A : b <_T c \}$ is in fact uncountable, contradicting that $T$ is an 
	Aronszajn tree.

	Now we are ready to prove the lemma. 
	Let $G$ be a generic filter on $\p_T$ which contains $a$. 
	We will show that in $V[G]$, $A$ contains an uncountable chain and an uncountable antichain. 
	Now $G$ is a chain of $T$, and $G \cap A$ is uncountable by the first paragraph. 
	So $G \cap A$ is an uncountable chain which is a subset of $A$, which proves the first part. 
	Since $G$ is a chain, 
	the claim in the previous paragraph easily implies that for each $b \in G \cap A$ above $a$, 
	there exists $c_b >_T b$ such that $c_b$ is in $A \setminus G$. 
	Now we can recursively define a sequence $\langle b_i : i < \omega_1 \rangle$ satisfying:
	\begin{enumerate}
		\item for all $i < \omega_1$, $b_i >_T a$ and $b_i \in G \cap A$;
		\item for all $\delta < \omega_1$, for all $i < \delta$, $\h_T(b_\delta) > \h_T(c_{b_i})$.
		\end{enumerate}
	
	We claim that $\{ c_{b_i} : i < \omega_1 \}$ is an 
	uncountable antichain contained in $A$. 
	Suppose for a contradiction that there are $i < j < \omega_1$ such that 
	$c_{b_i}$ and $c_{b_j}$ are comparable. 
	Since $b_j <_T c_{b_j}$ and 
	$\h_T(c_{b_i}) < \h_T(b_j)$, $c_{b_i} <_T c_{b_j}$, 
	and hence also $c_{b_i} <_T b_j$. 
	Since $b_j \in G$ and $b_j \le_{\p_T} c_{b_i}$, 
	$c_{b_i} \in G$, 
	which contradicts the choice of $c_{b_i}$.
	\end{proof}

\begin{lemma}
	Let $T$ be a normal $\omega_1$-tree. 
	Let $\vec c \in T^n$ be injective and assume that $T_{\vec c}$ is Suslin. 
	Suppose that $\langle \vec a_i : i < \omega_1 \rangle$ is a sequence of 
	$n$-tuples above $\vec c$ with increasing height. 
	Let $f, g : n \to 2$ be the functions such that $f(i) = 1$ and $g(i) = 0$ for all $i < n$. 

	Then the sequence $\langle \vec a_i : i < \omega_1 \rangle$ realizes both $f$ and $g$.
	Moreover, there exists $\vec d \in T_{\vec c}$ 
	such that $\vec d$ forces in $\p_{T_{\vec c}}$ 
	that there exist uncountable sets 
	$X, Y \subseteq \omega_1$ 
	such that for all $\xi < \beta$ in $X$, 
	$\vec a_\xi$ and $\vec a_\beta$ satisfy the type $f$, and 
	for all $\gamma < \delta$ in $Y$, $\vec a_\gamma$ and $\vec a_\delta$ satisfy the type $g$.
	\end{lemma}

\begin{proof}
	Note that the second conclusion is stronger than the first, since satisfying a type 
	is absolute between models of set theory. 
	We will prove the lemma by induction on $n$. 
	For the base case, let $c \in T$ and assume that $T_c$ is Suslin. 
	Let $\langle a_i : i < \omega_1 \rangle$ be a sequence 
	of nodes with increasing height above $c$. 
	By Lemma 5.1, we can find $d \in T_c$ such that $\{ a_i : i < \omega_1 \}$ is dense below $d$ in $\p_{T_c}$. 
	By Lemma 5.2, $d$ forces in $\p_{T_c}$ that 
	$\{ a_i : i < \omega_1 \}$ contains an uncountable chain and an uncountable antichain. 
	But any two distinct members of a chain will satisfy $f$, and any two distinct members of 
	an antichain will satisfy $g$.
	
	For the inductive step, fix $1 \le n < \omega$ and assume that the desired property holds for $n$. 
	Let $\vec c = (c_0,\ldots,c_n)\in T^{n+1}$ be injective and assume that $T_{\vec c}$ is Suslin. 
	Consider a sequence of $(n+1)$-tuples  
	$\langle \vec a_i : i < \omega_1 \rangle$ with increasing height above $\vec c$.

	Now $T_{\vec c \restrict n}$ is Suslin and $\langle \vec a_i \restrict n : i < \omega_1 \rangle$ is a 
	sequence of $n$-tuples with increasing height above $\vec c \restrict n$. 
	By the inductive hypothesis, there exists $\vec d^* \in T_{\vec c \restrict n}$ and names $\dot X^*$ 
	and $\dot Y^*$ for uncountable subsets of $\omega_1$ 
	such that $\vec d^*$ forces 
	in $\p_{T_{\vec c \restrict n}}$ that 
	for all $\xi < \beta$ in $\dot X^*$, 
	$\vec a_\xi \restrict n$ and $\vec a_\beta \restrict n$ satisfy the type $f \restrict n$, 
	and for all $\gamma < \delta$ in $\dot Y^*$, 
	$\vec a_\gamma \restrict n$ and $\vec a_\delta \restrict n$ 
	satisfy the type $g \restrict n$.
	
	Let $G^*$ be a generic filter on $T_{\vec c \restrict n}$ which contains $\vec d^*$. 
	Let $X^* := (\dot X^*)^G$ and $Y^* := (\dot Y^*)^G$. 
	Since $T_{\vec c}$ is Suslin, $T_{c_n}$ remains a Suslin tree in $V[G^*]$ by Lemma 1.1. 
	Write $\vec a_\xi = (a_{\xi,0},\ldots,a_{\xi,n})$ for all $\xi < \omega_1$. 
	Now $\{ a_{\xi,n} : \xi \in X^* \}$ is an uncountable subset of $T_{c_{n}}$. 
	By Lemmas 5.1 and 5.2 applied in $V[G^*]$, 
	we can find $d >_T c_{n}$ and names $\dot X'$ and $\dot Y'$  
	for uncountable subsets of $X^*$ and $Y^*$ respectively such that $d$ 
	forces in $\p_{T_{c_{n}}}$ over $V[G^*]$ that 
	$\{ a_{\xi,n} : \xi \in \dot X' \}$ is a chain and 
	$\{ a_{\gamma,n} : \gamma \in \dot Y' \}$ is an antichain. 
	Then $d$ forces that if $\xi < \beta$ are in $\dot X'$, 
	then they are in $X^*$ and hence 
	$\vec a_{\xi} \restrict n$ and $\vec a_{\beta} \restrict n$ 
	satisfy the type $f \restrict n$, which is equivalent to saying that 
	$\vec a_{\xi} \restrict n <_{T_{\vec c \restrict n}} \vec a_{\beta} \restrict n$; 
	moreover, $a_{\xi,n} <_{T_{c_{n}}} a_{\beta,n}$. 
	So altogether, $d$ forces that $\vec a_\xi <_{T_{\vec c}} \vec a_{\beta}$, and hence 
	that $\vec a_\xi$ and $\vec a_\beta$ satisfy the type $f$. 
	A similar argument shows that $d$ forces that 
	if $\gamma < \delta$ are in $\dot Y'$, then $\vec a_\gamma$ and $\vec a_\delta$ satisfy the type $g$. 
	Now choose $\p_{T_{\vec c}}$-names $\dot X$ and $\dot Y$ for $\dot X'$ and $\dot Y'$ respectively 
	and a condition $\vec d = (d_0,\ldots,d_n)$ in $\p_{T_{\vec c}}$ with 
	$\vec d^* \le_{T^n} (d_0,\ldots,d_{n-1})$ and $d \le_T d_n$ 
	which forces the above information about $\dot X$ and $\dot Y$.
	\end{proof}

\begin{thm}
	Let $T$ be a normal Suslin tree. 
	Then for all $2 \le n < \omega$, $T$ is $n$-entangled iff $T$ is $n$-free. 
	In particular, $T$ is entangled iff $T$ is free.
	\end{thm}

\begin{proof}
	We already know that $n$-entangled implies $n$-free. 
	Let $2 \le n < \omega$ and assume that $T$ is $n$-free. 
	We will prove that $T$ is $n$-entangled. 
	Suppose that $\langle \vec a_i : i < \omega_1 \rangle$ is a sequence of 
	$n$-tuples with increasing height above some injective 
	$\vec c = (c_0,\ldots,c_{n-1})$ in $T^n$. 
	Let $g : n \to 2$. 
	We will find $\xi < \beta < \omega_1$ such that $\vec a_\xi$ and $\vec a_\beta$ satisfy $g$. 
	If $g$ is either constant $1$ or constant $0$, then we are done by Lemma 5.3, so assume not.

	Fix positive integers $m$ and $k$ such that $n = m + k$ and the set 
	$\{ i < n : g(i) = 1 \}$ has size $m$. 
	We first prove the statement in the simpler case that for all $i < m$, $g(i) = 1$, and for all 
	$m \le i < n$, $g(i) = 0$. 
	Since the property of a pair satisfying the type $g$ is absolute between models of set theory, 
	it suffices to find a generic extension of $V$ by a c.c.c.\ forcing in which 
	there exists an uncountable $X \subseteq \omega_1$ such that for all $\xi < \beta$ in $X$, 
	$\vec a_\xi$ and $\vec a_\beta$ satisfy $g$. 
	By Lemma 5.3, we can find $\vec d \in T_{\vec c \restrict m}$ and a 
	$\p_{T_{\vec c \restrict m}}$-name $\dot Y$ such that $\vec d$ forces that $\dot Y$ is 
	an uncountable subset of $\omega_1$ and for all $\xi < \beta$ in $\dot Y$, 
	$\vec a_{\xi} \restrict m$ and $\vec a_{\beta} \restrict m$ satisfy the constant $1$ function.
	
	Fix a generic filter $G$ on $\p_{T_{\vec c \restrict m}}$ which contains $\vec d$. 
	Let $\vec c \setminus m := (c_m,\ldots,c_{n-1})$. 
	Then in $V[G]$, $T_{\vec c \setminus m}$ is Suslin by Lemma 1.1. 
	And $\langle (a_{\xi,m},\ldots,a_{\xi,n-1}) : \xi \in \dot Y \rangle$ 
	is a sequence of $k$-tuples with increasing height above 
	$(c_m,\ldots,c_{n-1})$. 
	Applying Lemma 5.3 in $V[G]$, fix a $V[G]$-generic filter $H$ on $\p_{T_{\vec c \setminus m}}$ 
	and an uncountable set $X \subseteq Y$ in $V[G][H]$ so that for all $\xi < \beta$ in $X$, 
	$(a_{\xi,m},\ldots,a_{\xi,n-1})$ and $(a_{\beta,m},\ldots,a_{\beta,n-1})$ satisfy 
	the constant $0$ type. 
	Then in $V[G][H]$, for all $\xi < \beta$ in $X$, $\vec a_\xi$ and $\vec a_\beta$ 
	satisfy the type $g$.
	
	In general, fix a bijection $\sigma : n \to n$ so that $\sigma[m] = \{ i < n : g(i) = 1 \}$. 
	Define $g^* : n \to 2$ by $g^* := g \circ \sigma$. 
	Then $g^*$ satisfies that for all $i < m$, $g^*(i) = 1$, and for all $m \le i < n$, $g^*(i) = 0$. 
	Let $\vec e = (e_0,\ldots,e_{n-1})$ satisfy that for all $i < n$, 
	$e_i := c_{\sigma(i)}$. 
	For all $\xi < \omega_1$, let 
	$\vec b_\xi = (b_{\xi,0},\ldots,b_{\xi,n-1})$ be the $n$-tuple such that for all $i < n$, 
	$b_{\xi,i} := a_{\xi,\sigma(i)}$. 
	Then $\langle \vec b_\xi : \xi < \omega_1 \rangle$ is a pairwise disjoint sequence of $n$-tuples 
	with increasing height above $\vec e$. 
	By the previous paragraph, there exist $\xi < \beta < \omega_1$ such that $\vec b_\xi$ and $\vec b_\beta$ 
	satisfy the type $g^*$. 
	Unwinding the above definitions, 
	it is easy to check that $\vec a_\xi$ and $\vec a_\beta$ satisfy the type $g$.
	\end{proof}

This last theorem together with Theorem 4.4 implies the next result.

\begin{corollary}
	For all $1 \le n < \omega$, there exists an $\omega_1$-Knaster forcing which adds 
	a normal $n$-free Suslin tree such that 
	all derived trees of $T$ of dimension $n+1$ are special. 
	In particular, $n$-free does not imply $(n+1)$-free.
	\end{corollary}

We remark that G.\ Scharfenberger-Fabian previously proved that $n$-free does not imply 
$(n+1)$-free using a different argument (\cite[Corollary 3.5]{gido}), which solved 
an open problem of \cite{hamkins}.

\bibliographystyle{plain}
\bibliography{paper37}

\end{document}